\documentclass[a4paper]{article}

\usepackage{amssymb}

\usepackage{amsfonts,amsmath,amscd,amsthm,color}

\usepackage{tikz-cd}
\usepackage[linesnumbered,ruled]{algorithm2e}
\usepackage{enumitem}
\usepackage{changes}
\usepackage{yfonts}

\newfont{\eufm}{eufm10 scaled\magstep1}

\newcommand{\cC}{\mathcal{C}}

\newcommand{\cL}{\mathcal{L}}

\newcommand{\bbC}{\mathbb{C}}

\def\Spec{{\rm Spec}}
\def\dres{\partial{\rm Res}}

\def\ord{\rm ord}

\def\Ker{\rm Ker}

\def\SpF{\texttt{SpF}}

\def\sBsq{\textbf{Bsq}}
\def\Bop{\textsc{Bsq }}
\def\para{\vspace{1.5 mm}}
\def\Sing{\rm Sing}
\def\gcd{\rm gcd}

\newtheorem{thm}{Theorem}[section]
\newtheorem{lemma}[thm]{Lemma}
\newtheorem{cor}[thm]{Corollary}
\newtheorem{prop}[thm]{Proposition}
\newtheorem{defi}[thm]{Definition}
\newtheorem{rem}[thm]{Remark}

\newtheorem{ex}[thm]{Example}

\date{} %

\begin{document}

\title{Factoring Third Order Ordinary Differential Operators\\ over Spectral Curves}

\author{Sonia L. Rueda\\
Dpto. de Matem\' atica Aplicada, E.T.S. Arquitectura\\
Universidad Polit\' ecnica de Madrid.\\
Avda. Juan de Herrera 4, 28040-Madrid, Spain.\\
\tt{sonialuisa.rueda@upm.es}
\and
Maria-Angeles Zurro\\
Dpto. de Matem\' aticas\\
Universidad Aut\' onoma de Madrid\\
Campus de Cantoblanco, Madrid, Spain\\
\tt{mangeles.zurro@uam.es}
}

\maketitle

\thispagestyle{empty}

\begin{abstract}
We consider the classical factorization problem of a third order ordinary differential operator $L-\lambda$, for a spectral parameter $\lambda$.
  It is assumed that $L$ is an algebro-geometric operator, that it has a nontrivial centralizer, which can be seen as the affine ring of curve, the famous {\it spectral curve} $\Gamma$. In this work we explicitly describe the ring structure of the  centralizer of $L$ and, as a consequence, we prove that $\Gamma$  is a space curve. In this context, the first computed example of a non-planar spectral curve arises, for an operator of this type. Based on the structure of the centralizer, we give a symbolic algorithm, using differential subresultants, to factor $L-\lambda_0$ for all but a finite number of points $P=(\lambda_0 , \mu_0 , \gamma_0)$ of the spectral curve . 
\end{abstract}


\medskip


\section{Introduction}
The factorization of ordinary differential operators, from the point of view of symbolic computation, has attracted much attention at least for a couple of decades, see for instance \cite{BRONSTEIN1996, Fredet2003,  Hoeij1997-2, singer1996testing, Schwarz, Zhang2012}. A new approach was jet recently presented in \cite{MRZ1}, for the factorization of (second order)  algebro-geometric ordinary differential operators, equivalently operators having a nontrivial centralizer. It is indeed the centralizer, the set of all operators commuting with a given operator $L$, the structure that guaranties an effective factorization of $L-\lambda$, for an spectral parameter $\lambda$. Continuing with this line of work, in this occasion we will consider the effective factorization problem of $L-\lambda$ for an ordinary third-order differential operator
\begin{equation}\label{eq-operator-L}
L = \partial^ 3 + u_1 \partial + u_0,    
\end{equation}
with (stationary) potentials $u_0$, $u_1$ in a differential field $K$, with derivation $\partial$ and field of constants $\bbC$, the field of complex numbers.  
The potentials $u_0$ and $u_1$ will be assumed to be solutions of a stationary Boussinesq system \cite{DGU}. For short, we will call $L$ a {\sl Boussinesq operator}.


Boussinesq  systems have been widely studied, especially their rational solutions \cite{Clarkson2009, DGU, SC}. They generate a hierarchy of integrable equations, the {\sl Boussinesq hierarchy}, one of the Gelfand and Dickii integrable hierarchies of equations associated to differential operators of any order \cite{Dikii}. {The stationary version of the Boussinesq hierarchy ultimately gives families of differential polynomials, in the coefficients of $L$, that are conditions for the existence of a nontrivial operator $A$ commuting with $L$. In other words, Boussinesq operators have nontrivial centralizers, which are an essential ingredient of this work.}

The  Burchnall and Chaundy Theorem \cite{BC}, implies that Boussinesq operators are algebro-geometric differential operators \cite{wilson60algebraic}. This famous theorem establishes a correspondence between commuting differential operators and algebraic curves. The {\sf spectral curve}, classically defined by the so called Burchnall and Chaundy (BC) polynomial, allows an algebro-geometric approach to handling the direct and inverse spectral problems for the {\sf finite-gap} operators \cite{K77}. It is another famous result, Schur's Theorem \cite{Sch}, the one ensuring that centralizers have quotient fields that are function fields of one variable, therefore they can be seen as affine rings of curves, and in a formal sense these are {\it spectral curves} \cite{PRZ2019}.

Let us assume we are given a Boussinesq  operator $L$.
Our goal is to develop a factorization algorithm for $L-\lambda_0$, $\lambda_0\in\bbC$ as an operator in $K[\partial]$, for almost every point $P_0=(\lambda_0,\mu_0,\gamma_0)$ of the spectral curve $\Gamma$ of $L$. For this purpose, we have to establish an appropriate theoretical framework that we describe next. First we use Goodearl's results in \cite{Good} to give a precise description of the centralizer $\cC_K(L)$ of $L$ in $K[\partial]$, in Theorem \ref{thm-centralizer}. 

\para

{\bf Theorem A.} 
Let $L$ be a Boussinesq  operator in $K[\partial]$. 
Then $L$ has a nontrivial centralizer $\cC_K(L)$, that equals the free $\bbC[L]$-module of rank $3$ with basis $\{1, A_{1} , A_{2}\}$, with $A_i$ an operator of minimal order $3 n_i +i$, $n_i\geq 0$,  $i=1,2$. That is $\cC_K(L)=\bbC[L,A_1,A_2]$.

\para

The second part of the theoretical framework is to identify the spectral curve $\Gamma$ of the Boussinesq operator $L$. Considering the structure of the centralizer, we will prove that $\Gamma$ is an affine algebraic curve in $\mathbb{C}^3$, {in a generic situation  (where $A_2$ does not depend on $A_1$ ). This result is established for the first time, as far as we know. Spectral curves of Boussinesq operators are treated as planar curves in the existing literature (see for instance \cite{DGU}), but our results show that the planar curve situation is only a particular case}. 

We will prove that the defining ideal $I$ of $\Gamma$ is generated in $\bbC[\lambda,\mu,\gamma]$ by the BC polynomials $f_1, f_2, f_3$ of $L-\lambda$, $A_1-\mu$ and $A_2-\gamma$ taken pairwise. These polynomials will be computed using differential resultants. We call $L$ a {\sl geometrically reducible} Boussinesq operator, if its associated ideal $I=(f_1,f_2,f_3)$ is a prime ideal in $\mathbb{C}[\lambda,\mu,\gamma]$. The previous results are proved in Theorem \ref{thm-spectral curve} that can be also stated as follows.

\para

{\bf Theorem B.} 
Let $L$ be a geometrically reducible Boussinesq  operator in $K[\partial]$. 
Then $\cC_K(L)$ is isomorphic to the ring of the affine algebraic curve $\Gamma$ in $\bbC^3$ defined by the prime ideal $I=(f_1,f_2,f_3)$. More precisely, $ \cC_K(L)   \simeq {\bbC [\lambda, \mu ,\gamma]}/{I}$.

\para

We give the first example of a (geometrically reducible) Boussinesq operator with non-planar spectral curve in Example \ref{ex-alabeada}. As far as we know, it is the first example of a non-planar spectral curve explicitly computed.

\medskip

The third part of the theoretical framework consist of proving that, for almost every $P_0=(\lambda_0,\mu_0,\gamma_0)$ in $\Gamma$, all but a finite number of points in a set $Z$, then $L-\lambda_0, A_1-\mu_0, A_2-\gamma_0$ have a greatest common right factor $\partial+\phi_0$.
The differential resultant of two differential operators is a condition on their coefficients that guaranties a nontrivial right common factor. See \cite{MRZ1}, Section 3.2 for the definition and main properties of differential resultants and subresultants. We use differential subresultants to compute $\partial+\phi_i(\lambda,\mu,\gamma)$, $i=1,2,3$, the greatest common right divisor (\gcd), of $L-\lambda$, $A_1-\mu$ and $A_2-\gamma$ pairwise. The finite set $Z$ is thus the natural one, containing the singular points of $\Gamma$ and the points where each $\phi_i$ is not well defined. 
We prove in Proposition \ref{prop-factor-at-P0} that for almost every point $P_0=(\lambda_0,\mu_0,\gamma_0)$ of $\Gamma$, then $L-\lambda_0$ admits a {\sl spectral factorization}, given by the right factor
\[\partial+\phi_0=\partial+\phi_i(P_0).\]
In general, $\Gamma$ is not a smooth curve. Its singular points support pathological factorizations, that we will address in future works. We  establish in the following theorem the generic behavior of the factorization.

\medskip

{\bf Theorem C.}\label{thm-intro-alg-correctness}
Let us consider a geometrically reducible Boussinesq operator $L$ in $K[\partial]$ and $\lambda_0 \in\bbC$. Then for every $P_0=(\lambda_0,\mu_0,\gamma_0)$ in \ $\Gamma\backslash Z$ the first order operator $\partial+\phi_0=\gcd (L-\lambda_0, A_1 -\mu_0,A_2-\gamma_0)$ provides the spectral factorization,
\begin{equation*}
    L -\lambda_0 =(\partial^2 +\phi_0 \partial+\phi_0^2+2\phi_0'+u_1)\cdot (\partial+\phi_0).  
\end{equation*}

Finally the theoretical framework is ready to state the Spectral Factorization (\SpF) Algorithm \ref{alg-SpF}, with input list $[L, \lambda_0 ]$ and output list $\left[ I, P_0, \partial+\phi_0 \right]$, whenever a point $P_0=(\lambda_0,\mu_0,\gamma_0)$ of \ $\Gamma\backslash Z$ exists and the defining ideal $I$ of the spectral curve is prime.
The performance of the algorithm is illustrated in Example \ref{ex-alabeada}. This algorithm was implemented in {\bf Maple 20}.

\para 

{\it The paper is organized as follows}.
In Section \ref{sec-Bsq-hierarchy} we present Boussinesq differential systems, in a convenient way to study the factorization problem of a third order Boussinesq operator $L$. It also contains Theorem \ref{thm-centralizer} on the structure of the centralizer $\cC_K (L)$ as a free $\bbC[L]$-module. These modules have also a ring structure, and in Section \ref{sec-Differential Resultant} we prove Theorem \ref{thm-spectral curve}, where we compute the defining ideal of the spectral curve $\Gamma= \Spec{\left( \cC_K (L)\right)}$ in $\mathbb{C}^3$.  Section \ref{sec-factoring-Bsq} includes two main achievements of our work. First the theoretical results on the right factor of $L-\lambda_0$ are given in Proposition \ref{prop-factor-at-P0} and Theorem \ref{thm-alg-correctness}; second, a symbolic algorithm, \texttt{SpF} Algorithm \ref{alg-SpF}, that computes the first order right factor of $L-\lambda_0$, for all but a finite number of points of $\Gamma$. Moreover it includes Example \ref{ex-alabeada}, which is an important contribution of this article, as far as we know, the first effectively computed example of a non planar spectral curve.

\medskip

\section{Boussinesq Hierarchy and Centralizers}\label{sec-Bsq-hierarchy}

In this section we first define the Boussinesq differential systems, adapting the presentation given in \cite{DGU}, where the operator~$L$ is written as
\begin{equation}\label{eq-BsqDGU-L3}
 L =\partial^3 +q_1 \partial +\frac12q_1'+q_0.
\end{equation}
We denote $q_j'=\partial(q_j)$ and $q_j^{(n)}=\partial^n(q_j)$, $n\geq 1$.

\para

Using the notation of \cite{DGU}, we consider a differential recursion for two sequences of differential polynomials $\{f_{n,i}\}_{n\geq 0}$ and $\{g_{n,i}\}_{n\geq 0} $, $i=1, 2$. Precisely,   
\begin{equation*}
\left\{
\begin{array}{ll}
     3\ \partial (f_{n , i})&= 2\partial^3 (g_{n-1,i})+2 q_1 \partial (g_{n-1,i})+q_{1}' g_{n-1,i}\\
     &+3 q_0 \partial( f_{n-1,i})+2 q_{0}'f_{n-1,i},\\
     3\ \partial (g_{n,i})&= 3 q_{0} \partial (g_{n-1,i})  + q_{0}'g_{n-1,i}-\frac{1}{6}\partial^5 (f_{n-1 ,i})\\
      &-\frac{5}{6}q_1 \partial^3 (f_{n-1 ,i})  -\frac{5}{4}q_1'\partial^2 (f_{n-1,i})\\
      &-\left(
      \frac{3}{4}q_1''+\frac{2}{3} q_1^2
      \right)\partial (f_{n-1,i})-
      \left(
       \frac{1}{6}q_1''' +\frac{2}{3}q_1 q_1'
       \right) f_{n-1,i} ,
\end{array}
\right.
\end{equation*}
with initial conditions: $(f_{0,1}, g_{0,1})=(0,1) $ and $(f_{0,2},g_{0,2})=(1,0)$. 

\para

Next, adapting \cite{DGU}, formula (5.5), with zero integration constants, we consider a family of differential operators associated to the operator $ L $ given in \eqref{eq-BsqDGU-L3}.
We define the \Bop differential operators
\begin{equation}\label{defi-Bsq-formal-operators}
P_i=L_{0,i},\,\,\, P_{3n+i}=P_{3n-3+i}L_3+L_{n,i} \ , i=1,2,   
\end{equation}
where 
\begin{equation*}
 L_{n,i}=f_{n,i} \partial^2 + 
    \left( g_{n,i} -\frac{1}{2} \partial (f_{n , i}) 
    \right) \partial 
    +
    \left(
    \frac{1}{6} \partial^2 (f_{n,i}) -\partial (g_{n,i}) +
    \frac{2}{3} q_1 f_{n,i}\right).    
\end{equation*}

In previous notations, see \cite{DGU}, we define the {\sl $n$th stationary Boussinesq system, in the $i$th-branch, of the Boussinesq hierarchy} as
\begin{equation}\label{eq-Bsq-system}
\sBsq_{n,i} := 
3\left(\ 
 \partial (f_{n+1 , i}),\,
 \partial (g_{n+1 , i})\ 
\right) \ .
\end{equation}

\begin{rem}\label{rem-ctes}
Observe that $\sBsq_{n,i}$ is a systems of differential polynomials in $\bbC[{\bf c}]\{{\bf u}\}$, with a set of algebraic variables ${\bf c}=\{c_0,\ldots ,c_n\}$, called {\sl integration constants} and a set of differential variables ${\bf u}=\{u_0,u_1\}$, with $u_1 =q_1$ and $u_0= \frac12 q_1' +q_0$. As a consequence, $P_{3n+i}$ are generic differential operators in $\bbC[{\bf c}]\{{\bf u}\}[\partial]$. We emphasize the dependence of potentials and constants as follows
\begin{equation}
 \sBsq_{n,i} =\sBsq_{n,i} (u_0 ,u_1 , {\bf c}) \ , \ \textrm{and }\quad    P_{3n+i}= P_{3n+i}(u_0 ,u_1 , {\bf c}). 
\end{equation}
\end{rem}

\begin{ex}\label{ex-Bsq-symbolic-eq}
Let $m$ be a positive integer not divisible by $3$, that is $m=3n+i$ for  $i=1,2$. For $n=0,1$  we write next the first four stationary Boussinesq systems.
\begin{enumerate}
    \item For $n=0$, we have  the following stationary Boussinesq systems:
    \begin{equation}
        \sBsq_{0,1}= (q_1', q_0' ) ,
    \end{equation}
    \begin{equation}
        \sBsq_{0,2}=\left(2 q_0' ,
        -\frac16 q_1''' -\frac23 q_1 q_1'\right)+c_0 \sBsq_{0,1} .
    \end{equation}
For $m$ equal $1,2$ the corresponding $\Bop$  operators read
    \[
    P_{1}=\partial \  \textrm{ and } \  P_{2}= \partial^2+\frac23 q_1+c_0 P_1\ .
    \]
    \item For $n=1$ we write below the stationary Boussinesq systems $\sBsq_{1,1}$ and $\sBsq_{1,2}$. In this case $m=4,5$ and we obtain the $\Bop$ operators: $P_{4}$ and $P_5$.
    \begin{align*}
        \sBsq_{1,1}= (b_{11}^1 ,b_{11}^2) +&c_1 \sBsq_{0,2}(q_0 ,q_1 ,0)+c_0 \sBsq_{0,1} ,\\
        \sBsq_{1,2}=(b_{12}^1 ,b_{12}^2)+
       &c_2 \sBsq_{1,1} (q_0 ,q_1, (0,0)) 
       + c_1 \sBsq_{0,2}(q_0 ,q_1 ,0)+\\  &c_0 \sBsq_{0,1} ,
    \end{align*}
   
 with
 \begin{align*}
      b_{11}^1 =&\frac23 q_0''' + \frac43 q_1 q_0' +\frac43 q_0 q_1',\\
      b_{11}^2 =&-\frac{q_1^{(5)}}{18}-\frac13 q_1 q_1''' -\frac23 q_1'q_1''-\frac49 q_1^2 q_1' +\frac43 q_0 q_0' ,\\
      b_{12}^1 =& -\frac{ q_1^{(5)}}{9} -\frac59 q_1 q_1'''  -\frac{25}{18} q_1' q_1''
 -\frac59 q_1^2 q_1' +\frac{10}{3} q_0 q_0' \\
      b_{12}^2 =&\frac19 q_0^{(5)} +\frac{5}{18} q_0 q_1''' +\frac59 q_1 q_0'''+\frac59 q_1'' q_0'+\frac56 q_1' q_0''\\
 &+ \frac59 q_1^2 q_0' +\frac{10}{9} q_0 q_1 q_1' ,
 \end{align*}
\begin{align*}
    P_{4}=&\partial^4+\frac43 q_1\partial^2+\left(\frac43 q_0+\frac43 q_1'\right)\partial+\frac59 q_1'' + \frac23 q_0'+\frac29 q_1^2+c_1P_2+c_0P_1\\
    P_5=&\partial^5+\frac53 q_1\partial^3+\left(\frac53 q_0+\frac52 q_1' \right)\partial^2+\left( \frac59 q_1^2+\frac{35}{18}q_1''+\frac53 q_0' \right)\partial\\
    &+\frac{10}{9}q_0''+\frac59 q_1q_1''\frac59 q_1'''+\frac{10}{9}q_1q_0+c_2P_4+c_1P_2+c_0P_1
\end{align*}
\end{enumerate}
\end{ex}

\medskip
In the remaining parts of the section, we establish the algebraic structure of the centralizer of $L=\partial^3+\tilde{u}_1\partial+\tilde{u}_0$, with $\tilde{u}_0$ and $\tilde{u}_1$ satisfying one of the Boussinesq systems of the Boussinesq hierarchy (Theorem \ref{thm-centralizer}). We will illustrate these results computing some examples.

\begin{defi}\label{defi-Bsq-operators}
We call a  differential operator $L=\partial^3+\tilde{u}_1\partial+\tilde{u}_0$ in $K[\partial]$  a Boussinesq operator if $\tilde{u}_0, \tilde{u}_1\in K$ verify a Boussinesq system for some $s$ and a choice of integration constants $\tilde{\bf c}=(\tilde{c}_0,\ldots ,\tilde{c}_s)$ in $\bbC^{s+1}$.
That is, for $i$ equal $1$ or $2$
\begin{equation}\label{eq-Bsq-poly}
\sBsq_{s,i} ( \tilde{u}_0 , \tilde{u}_1,\tilde{\bf c} )=(0,0)
    \ .
\end{equation}
\end{defi}

In addition, the following lemma shows that $L=\partial^3+\tilde{u}_1\partial+\tilde{u}_0$ is a Boussinesq operator if and only if $P_{3s+i} ( \tilde{u}_0 , \tilde{u}_1,\tilde{\bf c} )$ belongs to the centralizer of $L$.

\begin{lemma} \label{lemma-system}
With the previous notation, given $L=\partial^3+\tilde{u}_1\partial+\tilde{u}_0$ in $K[\partial]$, the following equivalence holds for each vector $\tilde{\bf c}=(\tilde{c}_0,\ldots ,\tilde{c}_n)$ in $\bbC^{n+1}$, $n\geq 0$, $i=1,2$:
\begin{equation*}
    \sBsq_{n,i} ( \tilde{u}_0 , \tilde{u}_1,\tilde{\bf c} )=(0,0)
    \Longleftrightarrow  
        [P_{3n+i} ( \tilde{u}_0 , \tilde{u}_1,\tilde{\bf c} ), L]=0\textrm{ in }  K[\partial] .
\end{equation*}
\end{lemma}
\begin{proof}
The statement follows by \cite{DGU}, (5.6), namely
\[[P_{3n+i},L]=3\partial(f_{n+1,i})\partial+3\left(\frac12\partial^2 (f_{n+1,i})+\partial(g_{n+1,i})\right).\]
\end{proof}


{\begin{ex}\label{ex-L3-first}
Let us consider the differential field of rational functions with complex coefficients $K=\bbC (x)$, and the differential operator in $K[\partial ]$
\begin{equation}
     L =\partial^3 -\frac{15}{x^{2}} \partial +\frac{15}{x^{3}}+h ,
\end{equation}
obtained from formula \eqref{eq-BsqDGU-L3} for $ q_0 = h$ and $ q_1 = -\frac{15}{x^{2}}$, with $h\in \mathbb{C}$. {For $h=0$ this is Example 8.3 (ii) in \cite{DGU}.}
The first Boussinesq systems 
satisfied in each branch $i=1,2$ are:
\[
\sBsq_{1,1}\left( 15/x^3 +h, -15/x^2 ,\tilde{c}_{1,1}\right)=0 , \mbox{ for }\tilde{c}_{1,1}=\left(-\frac43 h,0\right)
\]
and
\[
\sBsq_{2,2}\left( 15/x^3 +h, -15/x^2 ,\tilde{c}_{2,2}\right)=0, \mbox{ for }\tilde{c}_{2,2}=\left(0,\frac{20}{9}h^2,0,-\frac83 h,0\right).
\]
Hence, by Lemma \ref{lemma-system}, $[L, P_4 ]=0$ and $[L,P_8 ]=0$ where $P_4$ and $P_8$ are given in  \eqref{defi-Bsq-formal-operators}, substituting $u_0$ by $ \tilde{u}_0=15/x^3 +h $ and $u_1$ by $ \tilde{u}_1= -15/x^2 $. 

Moreover, we have the commuting operators:
\[
P_4(\tilde{u_0},\tilde{u_1},\tilde{c}_{1,1}) = \partial^4 -\frac{20}{x^2 }\partial^2 +\frac{40}{x^3 }\partial , \mbox{ and }
\]
\[
P_8(\tilde{u_0},\tilde{u_1},\tilde{c}_{2,2}) =  \partial^8 -\frac{40}{x^2 }\partial^6 +\frac{240}{x^3 }\partial^5- \frac{800}{x^4 }\partial^4+\frac{1600}{x^5 }\partial^3-\frac{1600}{x^6 }\partial^2.
\]
Observe that 
$
\sBsq_{1,2}\left( 15/x^3 +h, -15/x^2 ,\tilde{\bf c}\right)\not= 0 , 
$
 for any choice of vector of constants $\tilde{\bf c}$.
\end{ex}
}

\para

From now on, we assume $L=\partial^3+\tilde{u}_1\partial+\tilde{u}_0$ to be a Boussinesq operator, whose potentials $\tilde{u}_0, \tilde{u}_1\in K$ verify a stationary Boussinesq system for some $s$ and a choice of integration constants $\tilde{\bf c}=(\tilde{c}_0,\ldots ,\tilde{c}_s)$ in $\bbC^{s+1}$. We will proceed next to analyze the algebraic structure that will allow factoring $L-\lambda_0$ for almost all $\lambda_0\in\bbC$. This structure is the centralizer $\cC_K(L)$ of $L$ in $K[\partial]$, as a $\bbC [L]$-submodule, namely
\begin{equation*}
    \cC_K(L)=\{A\in K[\partial]\mid LA=AL\}.
\end{equation*}

By \cite{Good}, Theorem 4.2, $\cC_K(L)$ is a commutative integral domain. Moreover, by \cite{Good}, Theorem 1.2, $\cC_K(L)$ is a free $\bbC[L]$-module with basis $\{A_i\mid Y\subseteq \{0,1,2\}\}$, where $i\in Y$ if there exists $A_i\in \cC_K(L)$ of order $\ord(A_i)=3n_i+i$, with $n_i\geq 0$ minimal for this condition. In particular $A_0=1$. Moreover, the rank of $\cC_K(L)$ as a free $\bbC[L]$-module is a divisor of $3$. Observe that an operator can have a trivial centralizer, that is $\cC_K(L)=\bbC[L]$, of rank $1$ as a $\bbC[L] $-module. We obtain the following result.

\begin{thm}\label{thm-centralizer}
Let $L=\partial^3+\tilde{u}_1\partial+\tilde{u}_0$ be a Boussinesq  operator as defined in \ref{defi-Bsq-operators}. Then $L$ has a nontrivial centralizer in $K[\partial]$ that equals the free $\bbC[L]$-module of rank $3$ with basis $\{1, A_{1} , A_{2}\}$, that is
\begin{equation}
 \begin{array}{cl}
   \cC_K(L)   &  =\{p_0(L)+p_1(L)A_{1}+p_2(L)A_{2}\mid p_i \in \bbC[L]\}\\
    \  & = \bbC[L]\langle 1,A_{1},A_{2 }\rangle,
 \end{array}   
\end{equation}
with $A_i$ an operator of minimal order $3 n_i +i$, $i=1,2$ and $n_i\geq 0$.
\end{thm}
\begin{proof}
By Lemma \ref{lemma-system}, a Boussinesq  operator has a nontrivial centralizer since $P_{3s+i}( \tilde{u}_0 , \tilde{u}_1,\tilde{\bf c})$ belongs to the centralizer for $s\geq 0$, $i$ equal $1$ or $2$ and  $\tilde{\bf c}\in \bbC^{s+1}$.
In addition, we know that the rank of $\cC_K(L)$ as a free $\bbC[L]$-module is a divisor of $3$, thus its basis must be $\{1,A_1,A_2\}$, with $A_i$ as defined above.
\end{proof}

By the previous theorem the centralizer of $L$ is the domain
\begin{equation*}
    \cC_K(L)=\bbC[L,A_{1},A_{2}].
\end{equation*}
For any $L$ in $K[\partial ]$, centralizers $\cC_K(L)$ are maximal commutative subrings.
So, given a differential operator $M$ that commutes with $L$, we have the sequence of inclusions 
\begin{equation*}
{\bf C}[L]\subseteq {\bf C}[L,M] \subseteq \cC_{K}(L) ,  
\end{equation*} 
and all of them could be strict. In the case of a Boussinesq operator $L=\partial^3+\tilde{u}_1\partial+\tilde{u}_0$, in notations of Theorem \ref{thm-centralizer}, we have the following ring diagram.
\begin{equation}\label{diagram}
\begin{tikzcd}
                                         &  & {\bbC[L ,A_{1}]} \arrow[rrd]   &  &                  \\
{\bbC[L ]} \arrow[rru]  \arrow[rrd] &  & {\bbC[A_{1} ,A_{2}]} \arrow[rr] &  & {\bbC[L ,A_{1} ,A_{2} ]} \\
                                          &  & {\bbC[L ,A_{2}]} \arrow[rru]   &  &   
\end{tikzcd}
\end{equation}

\para 

We give next an example of the generators of the centralizer of a  Boussinesq operator with rational coefficients.


\begin{ex}\label{ex-centralizer-Bsq}
Continuing with Example \ref{ex-L3-first}, we can compute the centralizer of the Boussinesq operator
\begin{equation*}
     L =\partial^3 -\frac{15}{x^{2}} \partial +\frac{15}{x^{3}}+h \ .
\end{equation*}
With the notation of the previous theorem, we have $A_1 = P_4$ and $A_2 = P_8$. It should be noted that there is no $5$-th order operator in the centralizer $\cC_K(L)$ because
$
\sBsq_{1,2}\left( 15/x^3 +h, -15/x^2 ,\tilde{\bf c}\right)\not= 0 , 
$ for any choice of vector of constants $\tilde{\bf c}\in \mathbb{C}^3$.
\end{ex}

\section{Differential  resultants and spectral curves}\label{sec-Differential Resultant}


A polynomial $f(\lambda,\mu)$ with constant coefficients satisfied by a commuting pair of differential operators $P$ and $Q$ is called a {\it Burchnall-Chaundy  (BC) polynomial} of $P$ and $Q$, since the first result of this sort appeared is the $1923$ paper \cite{BC} by Burchnall and Chaundy. Therefore, associated to the centralizer of a Boussinesq operator $L$ there are as many BC polynomials as operators in the centralizer. By E. Previato's Theorem \cite{MRZ1}, BC polynomials can by computed using differential resultants, and we show next how to compute them for a Boussinesq operator $L$. 

Given a differential operator $A\in \cC_K(L)$ of order $m$, the {\it differential resultant} of $L-\lambda$ and $A-\mu$ equals
$$\dres(L-\lambda,A-\mu):=\det (S_0(L-\lambda,A-\mu)),$$
where the Sylvester matrix $S_0(L-\lambda,A-\mu)$ is the coefficient matrix of the extended system of differential operators
\[\Xi_0=\{\partial^{m-1}(L-\lambda),\ldots \partial(L-\lambda), L-\lambda, \partial^2(A-\mu),\partial (A-\mu), A-\mu\}.\]
Observe that $S_0(L-\lambda,A-\mu)$ is a squared matrix of size $3+m$ and entries in $K[\lambda,\mu]$. See \cite{MRZ1}, Section 3.2.1 for the definition and main properties of differential resultants and \cite{MRZ1}, Section 5.2 for a proof of Previato's Theorem, from which the next result immediately follows.

\begin{thm}
Let $L$ be a Boussinesq operator and let us consider a differential operator $A$ in its centralizer $\cC_K (L)$. The BC polynomial of $L$ and $A$ equals
\begin{equation}
    f(\lambda,\mu)=\dres(L-\lambda,A-\mu)\in C[\lambda,\mu].
\end{equation}
\end{thm}

Observe that the operators $L-\lambda$ and $A-\mu$ have coefficients in the differential ring $(K[\lambda,\mu],\partial)$ and,
by means of the differential resultant, it is ensured that we compute a nonzero polynomial
\begin{equation}\label{eq-dresnonzero}
\dres(L-\lambda,A-\mu)=a_3^m\mu^3-b_m^3\lambda+\cdots
\end{equation}
in the elimination ideal $(L-\lambda,A-\mu)\cap K[\lambda,\mu]$. Previato's \
Theorem implies that 
\[\dres(L-\lambda,A-\mu)\in (P-\lambda,Q-\mu)\cap C[\lambda,\mu].\]


We consider next the centralizer of $L$ as established in Theorem \ref{thm-centralizer}
$\cC_K (L)= \mathbb{C}[L,A_1 ,A_2 ]$.
Hence, applying Previato's Theorem and \eqref{eq-dresnonzero} pairwise in $\left\{L , A_{1} ,  A_{2}  \right\} $, with $\ord(A_1)=3n_1 +1$ and $\ord(A_2)=3n_2 +2$, we obtain the BC polynomials in $\bbC[\lambda,\mu,\gamma]$
\begin{align}\label{defi-f1-f2-f3}
&f_1(\lambda,\mu,\gamma)=\dres(L-\lambda,A_{1}-\mu)=\mu^3-\lambda^{3n_1 +1}+\cdots , \\
&f_2(\lambda,\mu,\gamma)=\dres(L-\lambda,A_{2}-\gamma)=\gamma^3-\lambda^{3n_2 +2}+\cdots, \\
&f_3(\lambda,\mu,\gamma)=\dres(A_{1}-\mu, A_{2}-\gamma)=\mu^{3n_2 +2}-\gamma^{3n_1 +1}+\cdots.
\end{align}

Observe that $f_i$, $i=1,2$ are irreducible polynomials, since $L$ and $A_{i}$ have coprime orders. The orders of $A_1$ and $A_2$ may not be coprime and the structure of the centralizer together with the nature of the spectral curve depend on this question. 

\begin{defi}\label{defi-Bsq-ideal}
Let $L=\partial^3+\tilde{u}_1\partial+\tilde{u}_0$ in $K[\partial]$ be a Boussinesq operator.
We define {\sl the ideal associated to $L$} to be the ideal in $\mathbb{C}[\lambda,\mu,\gamma]$ generated by the set $\{ f_1 , f_2 , f_3 \} $, with $f_i$ defined in \eqref{defi-f1-f2-f3}. We denote this ideal by $I(L)$. We call $L$ a {\sl geometrically reducible Boussinesq operator}, if its associated ideal $I (L)$ is a prime ideal in $\mathbb{C}[\lambda,\mu,\gamma]$.
\end{defi}

By Schur's Theorem \cite{Sch}, the quotient field of the centralizer is a function field in one variable, therefore it is the affine ring of a curve, the spectral curve $\Gamma$. 
{In the next theorem we explicitly define $\Gamma$ and the appropriate isomorphism.} 

\begin{thm}
\label{thm-spectral curve}
Let $L$ be a geometrically reducible Boussinesq operator. 
Then, the centralizer of \ $L$ in $K[\partial]$ is isomorphic to the ring of the affine algebraic curve $\Gamma$ in $\bbC^3$ defined by the prime ideal $I=I(L)$ associated to $L$. More precisely, we have a ring isomorphism
\begin{equation}\label{eq-isomorfismo}
    \cC_K(L)   \simeq {\bbC [\lambda, \mu ,\gamma]}/{I} \ .
\end{equation}
\end{thm}
\begin{proof} Let us consider the centralizer
$\cC_K (L)= \mathbb{C}[L,A_1 ,A_2 ]$.
By \cite{Eisenbud-AlgCom}, page $286$ Theorem A,
\begin{equation}\label{dim-centralizer}
   \dim\left(
    \cC_K (L )
    \right) =\textrm{tr. deg}_\bbC  Fr(\cC_K (L ))=1 ,
\end{equation}
and the length of every maximal chain of primes in $\cC_K (L ) $ is $1$.

We define the  homomorphism $ \psi: \bbC [\lambda ,\mu ,\gamma] \rightarrow \cC_K (L )$ by
\begin{equation}
   \psi (\lambda )=L \ , \quad  \psi (\mu )=A_{1} \ , \quad \psi (\gamma )=A_{2} \ .
\end{equation}
This map $\psi$ is a surjection, by Theorem \ref{thm-centralizer}. Its kernel $\textfrak{p}=\Ker \psi$ is a  prime ideal containing $I$. Then, assuming they are distinct ideals, the ideal
$\textfrak{p}/I$ is a non zero ideal of $\bbC [\lambda ,\mu ,\gamma] /I $ of height $1$, since the Krull dimension of $\cC_K (L ) $ is $1$, by \eqref{dim-centralizer}. 

Next, we can consider the chain of prime ideals
\[
(0)\subset (f_1 ) \subset I\subset \textfrak{p} \subset \textfrak{m} \ , \textrm{ for some prime } \textfrak{m} .
\]
Therefore, the Krull dimension of $\bbC [\lambda ,\mu ,\gamma] $  would be  $4$, which is is a contradiction. In consequence,  $\textfrak{p}=I$, 
and thus the required ring isomorphim \eqref{eq-isomorfismo} is obtained. 
\end{proof}

We would like to point out that in previous works \cite{DGU},  the algebraic relations considered on the operators $L$ and $A_i$ (BC polynomials) were only those given by $f_1$ or $f_2$. The novelty of
this work is based on considering the BC polynomial $f_3$ to assign a one dimensional domain, canonically associated with the operator $L$. After Theorem \ref{thm-spectral curve}, this domain is the centralizer of $L$, whose spectrum as a ring is the {\sl spectral curve} $ \Gamma $ of $L$.

\begin{rem}
\begin{enumerate}
\item Regarding the geometrically reducible hypothesis on $L$. Observe that assuming that the ideal $I(L)$ is prime implies that the order of $L$, $A_1$ and $A_2$ are coprime. Example \ref{ex-spectral-curve-Bsq} shows the need of this hypothesis.
\item Observe that, if a non constant coefficient operator $A_2$ of order $2$ belongs to $\cC_k(L)$ then $\cC_K(L)=\cC_K(A_2)=\bbC[L,A_2]$, which is isomorphic to the ring of the plane algebraic curve defined by $f_2$. In this case the  operator of minimal order $3n_1+1$ in $\cC_K(L)$ is $A_2^2$, implying that $f_3=(\mu-\gamma^2)^2$ and therefore $I(L)$ is not prime.
\end{enumerate}
\end{rem}

The following example shows a particular case that illustrates the need of the hypotheses in Theorem \ref{thm-spectral curve}.

\begin{ex}\label{ex-spectral-curve-Bsq}
In Example \ref{ex-L3-first}, we  computed the centralizer of the Boussinesq operator  
\begin{equation*}
     L =\partial^3 -\frac{15}{x^{2}} \partial +\frac{15}{x^{3}}+h \ .
\end{equation*}
We compute the generators of the ideal $I(L)=(f_1,f_2,f_3)$ associated to $L$, using differential resultants, to obtain
\begin{align*}
f_1 =&-\mu^3+(\lambda -h)^4,
f_2=-\gamma^3 + (\lambda-h)^8, 
f_3 = (\gamma-\mu^2)^4. \
\end{align*}
The ideal $I(L)$ is not a prime, since $f_3$ is not square free. Moreover, since $f_3$ is the BC polynomial of $A_1$ and $A_2$ we have that $A_2=A_1^2$, implying that the centralizer is the ring $\cC_K(L)=\bbC[L,A_1]$. Thus, in this case the centralizer is isomorphic to the ring $\bbC[\lambda,\mu]/(f_1)$ of the plane curve defined by $f_1=0$. 
\end{ex}

\section{Factoring Boussinesq operators}\label{sec-factoring-Bsq}

In this section we  consider a geometrically  reducible Boussinesq operator as in Definition \ref{defi-Bsq-ideal},
\begin{equation}\label{eq-L-geom-irred}
     L =\partial^3 +\tilde{u}_1 \partial +\tilde{u}_0 ,
\end{equation}
whose coefficients, $\tilde{u}_0 , \tilde{u}_1$, are  solutions of a stationary Boussinesq system. Then, by Theorem \ref{thm-centralizer} there exist operators $A_1$ and $A_2$, generators of the centralizer $\cC_K (L)=\bbC[L,A_1,A_2]$. Moreover, the spectral curve $\Gamma$ of $L$ defines a field, the fraction field  $K(\Gamma)$ of the domain $K[\lambda, \mu, \gamma ]/I(L)$, with $I(L)=(f_1 ,f_2 ,f_3 )$ the defining ideal of $\Gamma$.

The goal of this section is to factor $L-\lambda_0$, with $\lambda_0\in\bbC$, as an operator with coefficients in $K$, for almost every point $P_0 = (\lambda_0 , \mu_0 , \gamma_0 )$ of the spectral curve $\Gamma$. To achieve this goal we will use differential subresultants, using similar methods to the ones developed in a recent work for algebro-geometric Schr\"odinger operators in \cite{MRZ1} and \cite{MRZ2}.

The differential resultant of two differential operators is a condition on their coefficients that guaranties a right common factor, see \cite{MRZ1}, Proposition 3.4 (2). {Since $L$ geometrically reducible implies $\gcd(\ord(L),\ord(A_1),\ord(A_2))=1$}, as polynomials in $K[\lambda,\mu,\gamma]$, the differential resultants $f_i$, $i=1,2,3$ defined in \eqref{defi-f1-f2-f3} are nonzero irreducible polynomials in $\bbC[\lambda,\mu,\gamma]$. Given $P_0=(\lambda_0,\mu_0,\gamma_0)$ in $\Gamma$, then $L-\lambda_0$, $A_1-\mu_0$ and $A_2-\gamma_0$ are operators in $K[\partial]$. Since $f_i(P_0)=0$, then the differential resultants of the pairs $\{L-\lambda_0 ,A_1-\mu_0 \}$, $\{L-\lambda_0 ,A_2-\gamma_0 \}$ and $\{A_1-\mu_0 ,A_2-\gamma_0 \}$ are zero, implying they have nonzero right common factors in $K[\partial]$. We will compute these factors using differential subresultants.

Given $A\in\cC_K (L)$, to define the first differential subresultant of $L-\lambda$ and $A-\mu$, with $\ord(A)=m$, we need the coefficient matrix of the extended system of differential operators
\[\Xi_1=\{\partial^{m-2}(L-\lambda),\ldots \partial(L-\lambda), L-\lambda, \partial (A-\mu), A-\mu\}.\]
Observe that its coefficitent matrix $S_1(L-\lambda,A-\mu)$ is a matrix with $m+1$ rows, $m+2$ columns and entries in $K$.
The first differential subresultant of $L-\lambda$ and $A-\mu$ is the differential operator
\begin{equation}
\cL_1=\det(S_1^0)+\det(S_1^1)\partial    
\end{equation}
where
\begin{equation}
S_1^0:=      \textrm{submatrix}(S_1,\hat{\partial})  , 
\mbox{ and }S_1^1:=\textrm{submatrix}(S_1,\hat{1})  
\end{equation}
are the submatrices of $S_1 =S_1 (L-\lambda,A-\mu)$ obtained by removing columns indexed by $\partial$ and $1$ respectively.

For $j=0,1$, let us define the polynomials in $K[\lambda,\mu,\gamma]$
\begin{align}\label{eq-phiij}
    \phi_{1,j}(\lambda,\mu,\gamma):=&\det(S_1^j (L-\lambda ,A_1-\mu )),\\
    \phi_{2,j}(\lambda,\mu,\gamma):=&\det(S_1^j (L-\lambda ,A_2-\gamma )),\\
    \phi_{3,j}(\lambda,\mu,\gamma):=&\det(S_1^j (A_1-\mu ,A_2-\gamma )).
\end{align}
and the rational functions
\begin{equation}\label{eq-phiis}
    \phi_i(\lambda,\mu,\gamma)=\frac{\phi_{i,0}(\lambda,\mu,\gamma)}{\phi_{i,1}(\lambda,\mu,\gamma)},\,\,\, i=1,2,3.
\end{equation}
Let us consider the finite set 
\begin{equation}\label{eq-Z}
    Z=\Sing(\Gamma)\cup\{P_0\in \Gamma\mid \phi_{i,1}(P_0)=0\},
\end{equation}
union of the singular points of $\Gamma$ with the points for which the rational functions $\phi_i$ are not well defined.

\begin{prop} Let $L$ be a geometrically reducible Boussinesq operator in $K[\partial]$. For  every point $P_0 = (\lambda_0 , \mu_0 , \gamma_0 )$ in $\Gamma\backslash Z$, $i=1,2,3$, $j=0,1$, then $\phi_{i,j}(P_0)$ are non zero elements of $K$. 

\end{prop}
\begin{proof}
Recall that $\gcd(\ord(L),\ord(A_i))=1$ guaranties that $f_i$ are irreducible polynomials in $K[\lambda,\mu,\gamma]$. Thus the degree in $\lambda$ of $f_i$ is $\ord(A_i)$ and the degree in $\mu$ (or $\gamma$) is $3$. By the construction of the matrices $S_1^j$, the degree in $\mu$ (or $\gamma$) of $\phi_{i,j}(\lambda,\mu,\gamma)$, $i=1,2$ is less than or equal to $2$. Therefore $\phi_{i,j}$ does not belong to $I(L)$. Thus $\phi_{i,j}$ could only be zero for a finite number of points $P_0$ in $Z$.
\end{proof}

\para 

Let $L$ be a geometrically  reducible Boussinesq operator in $K [\partial]$.
We  give next the result that allows factoring the operator $L-\lambda_0$ for almost every point $P_0 = (\lambda_0 , \mu_0 , \gamma_0 )$ in $\Gamma$.

Let $P_0$ be a point of $\Gamma\backslash Z$. With the previous notation, we compute the
following (monic) greatest common (right) divisors, {evaluating \eqref{eq-phiis} in $P_0$}:
\begin{align}
 \cL_{1}=&\gcd \left( L-\lambda_0 , A_{1}-\mu_0 \right)= \partial +\phi_1 (P_0 ), \label{gcd-12} \\
 \cL_{2}=&\gcd \left( L-\lambda_0 , A_{2}-\gamma_0 \right)= \partial +\phi_2 (P_0 ), \label{gcd-13}\\
 \cL_{3}=&\gcd \left( A_{1}-\mu_0 , A_{2}-\gamma_0 \right)= \partial +\phi_3 (P_0 ) \label{gcd-23}.
\end{align}
Then, $\partial+\phi_i (P_0 )\in K[\partial ]$ and we obtain the equalities:
\[
L-\lambda_0 = N_i \cdot (\partial+\phi_i (P_0 )), \ i=1,2
\]
for second order differential operators $N_i$. 

\medskip

Let 
$\Sigma$ be a Picard-Vessiot field (with field of constants $\mathbb{C}$) for the differential operator $L-\lambda_0$. Since $\partial+\phi_i (P_0)$ is a divisor of $L-\lambda_0$, $i=1,2$ then (see for instance \cite{VPS}, page 15), there exists $(a_1 , a_2 )\in \Sigma^2$ satisfying the  system $\{F_1=0,F_2=0\}$ for differential polynomials 
\begin{equation*}
    F_1 = y'_{1} +\phi_1 (P_0 ) y_1 , \ F_2 = y'_{2} +\phi_2 (P_0 )y_2 ,
\end{equation*}
in the differential field $(K\{y_1 , y_2 \} , ')$ with field of constants  $\mathbb{C}$. Then, by the Weak differential Nullstellensatz, \cite{GUSTAVSON2016} 1142 
, $1$ is not in the differential ideal generated by $F_1$ and $F_2$. Therefore, since they are monic first order polynomials, we conclude that $\phi_1 (P_0 ) = \phi_2 (P_0 )$. But, because of \eqref{gcd-23}, the following equalities hold
\begin{equation}\label{eq-factor comun}
\partial+\phi_1 (P_0 )= \partial+\phi_2 (P_0 )=\partial+\phi_3  (P_0 )   \ ,
\end{equation}
 and \  $\partial+\phi (P_0 )=\partial +\phi_i (P_0 )$, \ {is the greatest common right divisor of  $L-\lambda_0$, $A_1-\mu_0$ and $A_2-\gamma_0$. }

We summarize the previous construction in the following proposition.

\begin{prop}\label{prop-factor-at-P0}
Let $L$ be a geometrically reducible Boussinesq operator, and $\Gamma $ its spectral curve. There exists a rational function $\phi\in K(\Gamma)$ such that, for  every point $P_0 = (\lambda_0 , \mu_0 , \gamma_0 )$ in $\Gamma\backslash Z$, the operator $L-\lambda_0$ has as right factor 
$\partial+\phi (P_0 )=\gcd\left(L-\lambda_0,A_1-\mu_0,A_2-\gamma_0\right)$.
\end{prop}

Moreover, the following formula can be easily verified in $K[\partial]$:
 \begin{equation}\label{eq-factorization}
    L -\lambda_0 =(\partial^2+\phi (P_0 ) \partial+\phi(P_0 )^2+2\phi(P_0 )'+\tilde{u}_1)(\partial+\phi(P_0 )).
\end{equation}

\begin{defi}
We call \eqref{eq-factorization} a {\sl spectral factorization of $L-\lambda_0$}, which is obtained at almost every point $P_0=(\lambda_0 , \mu_0 , \gamma_0 )$ of the spectral curve $\Gamma$ of $L$, all but a finite number of points in $Z$.
\end{defi}

The above results can be automated. We propose the Spectral Factorization  algorithm, or \SpF \ for to short, (Algorithm \ref{alg-SpF}) to perform the factorization of a geometrically reducible Boussinesq operator $L -\lambda_0$ as in \eqref{eq-L-geom-irred} at a point $P_0 =(\lambda_0 , \mu_0 ,\gamma_0)$ of its associated spectral curve $ \Gamma $. The output is the greatest common right divisor $\partial+\phi(P_0 )$, according to the equality  \eqref{eq-factor comun}, of the generators of the centralizer $ \cC_K (L) $ at the point $P_0 $ of $\Gamma$. 

\medskip

\begin{algorithm}
\caption {Spectral Factorization (\texttt{SpF})}
\label{alg-SpF}
\SetKwInput{KwData}{Input}
\SetKwInput{KwResult}{Output}
\SetAlgoLined 
\KwData{A Boussinesq operator $L$ as in \eqref{eq-L-geom-irred} and $\lambda_0\in \bbC$.} 
\KwResult{A list $ \left[ I(L), P_0=(\lambda_0,\mu_0,\gamma_0), \cL_1 \right]$, with: $I(L)$ the defining ideal of the spectral curve $\Gamma$ of $L$; the coordinates of a point $P_0$ of $\Gamma$; $\cL_1=\partial+\phi(P_0)$ the right $\gcd$ of $\{L-\lambda_0,A_1-\mu_0,A_2-\gamma_0\}$, for generators $A_1,A_2$ of the centralizer $ \cC_K (L)$, that verifies \eqref{eq-factorization}.
}
    Compute $A_{1}$ and $A_{2}$, by means of \Bop operators in \eqref{defi-Bsq-formal-operators}, such that $\cC_K(L)=\bbC[L,A_1,A_2]$.\label{step-A1A2}\\
    Compute $f_1:=\dres(L-\lambda,A_1-\mu)$,  $f_2:=\dres(L-\lambda,A_2-\gamma)$ and  $f_3:=\dres(A_1-\mu,A_2-\gamma)$.\label{step-resultants}\\
    Define $I(L):=( f_1 , f_2 , f_3 )$ the ideal of the spectral curve $\Gamma$ of $L$, by Theorem \ref{thm-spectral curve}. \label{Step-def-ideal}
    \\
   If $I(L)$ is not  prime, then \KwRet{"\textsl{$L$ is not geome\-tri\-cally reducible}"}.
     \\
    Compute pairwise the first differential subresultants of $\{L-\lambda,A_1-\mu,A_2-\gamma\}$ to obtain $\phi_{i,1}(\lambda,\mu,\gamma)\partial +\phi_{i,0}(\lambda,\mu,\gamma)$, $i=1,2,3$, see \eqref{eq-phiij}. \label{step-subresultants}\\
    Compute the finite set $Z=\Sing(\Gamma)\cup\{P_0\in \Gamma\mid \phi_{i,1}(P_0)=0\}$\\
    Compute  $P_0=(\lambda_0,\mu_0,\gamma_0)$ in $\Gamma$. \label{Step-not-in-Z}\\
   {If $P_0$ belongs to $Z$ , then \KwRet{"\textsl{a spectral factorization of $L-\lambda_0$ cannot be obtained}"}.}\\
    Define the rational function $\phi(\lambda,\mu):=\phi_{1,0}/\phi_{1,1}$ and compute $\phi_0:=\phi(P_0)$.\\
    Define the right factor $\cL_1 =\partial+\phi_0$.\\
    \KwRet{ $ \left[ I(L), P_0, \cL_1 \right]$ }. \label{step:3return}
\end{algorithm}

\medskip

The next result guarantees the correctness of the previous algorithm.

\begin{thm}\label{thm-alg-correctness}
The \texttt{SpF} Algoritm \ref{alg-SpF}, with input list $[L, \lambda_0 ]$, where $L=\partial^ 3 + \tilde{u}_1 \partial + \tilde{u}_0$ is a  Boussinesq operator over $K$, and $\lambda_0 \in\bbC$, \ returns  "\textsl{$L$ is not geome\-tri\-cally reducible}", if the defining ideal $I(L)$ of $\Gamma$ computed in Step \ref{Step-def-ideal} is not a prime ideal; if $L$ is geometrically reducible, it returns a list $\left[ I(L), P_0, \cL_1 \right]$, with $P_0=(\lambda_0,\mu_0,\gamma_0)$ a point of \ $\Gamma\backslash Z$, with $Z$ as in \eqref{eq-Z}, and  
\begin{equation*}\label{eq-factor-of-3}
    \cL_1 =\partial+\phi(P_0) = \gcd \left(L-\lambda_0,A_1-\mu_0,A_2-\gamma_0\right) \ ,
\end{equation*}
for centralizer $ \cC_K (L)=\bbC[L,A_1,A_2]$.
{Moreover, at almost every point $P_0$ of \ $\Gamma$ (all but a finite number in $Z$) the spectral factorization  \eqref{eq-factorization} of $L-\lambda_0$ in $K[\partial ]$ is achieved.} 
\end{thm}
\begin{proof}
The correctness of this algorithm is an immediate consequence of Theorem \ref{thm-spectral curve} and Proposition \ref{prop-factor-at-P0}. Observe that $\cL_1$ is the right factor of Proposition \ref{prop-factor-at-P0}, and by the equality in \eqref{eq-factor comun} it is the greatest common right divisor of $L-\lambda_0 $, $A_1-\mu_0$ and $,A_2-\gamma_0$. Consequently, the formula \eqref{eq-factorization} is obtained.
\end{proof}

\begin{rem}
\begin{enumerate}
    \item Step \ref{step-A1A2}. 
    To compute the generators $ A_1 $ and $ A_2 $ of the centralizer of $ L $, we have implemented the recursion formula \eqref{defi-Bsq-formal-operators}. This formula is linked to the coefficients of the  operator $ L $,  the Boussinesq potentials, and the integration constants associated with them. By means of the stationary Boussinesq system \eqref{eq-Bsq-system} the vectors of integration constants can be previously  computed, as we did to construct the examples contained in this article. 
    
    The effective calculation of $ A_1 $ and $ A_2 $ is an interesting problem. The development of algorithms for their computation is part of {an ongoing project} 
    focused on the study of the centralizers of operators with potentials verifying the Gel'fand-Dickii integrable hierarchies \cite{GD}.
   
    \item Step \ref{Step-not-in-Z}. Calculating a point in $ \Gamma \setminus Z $ can be tricky. It should be noted that for the given value $ \lambda_0 $ the polynomial $ \phi_ {1,1} (\lambda_0, \mu) $ is a polynomial in $ K [ \mu] $ of degree at most $ 2 $ in $\lambda$. Consequently, it has at most two roots in $ \mathbb{C} \subset K $. Therefore, only a finite number of checks should be performed for each $ \mu_0 $ with $ f_1 (\lambda_0, \mu_0) = 0 $. An analogous situation occurs for $ \gamma_0 $ such that $ f_2 (\lambda_0, \gamma_0) = 0 $. Therefore, Step \ref{Step-not-in-Z} can be verified in a finite number of checks.
\end{enumerate}
\end{rem}

\subsection{Computed examples}

There are many examples of rank $1$ operators whose centralizers are the ring of a plane algebraic curve, see references in \cite{PRZ2019}. We show in Example \ref{ex-alabeada} the first explicit example of a centralizer isomorphic to the ring of a non-planar spectral curve. We use  Example \ref{ex-alabeada} to illustrate the performance of the \SpF Algorithm, to obtain the spectral factorization of $L-\lambda_0$, at almost every point (all but a finite number) $P_0 =(\lambda_0 ,\mu_0 ,\gamma_0)$ of the spectral curve $\Gamma$.

We also include Example \ref{ex-factoring-curva-plana}, for which the centralizer is isomorphic to the ring of a plane spectral curve. In this case the \SpF Algorithm cannot be applied but we can still obtain a right factor of $L-\lambda_0$ at every point $P_0 =(\lambda_0 ,\mu_0 ,\gamma_0)$ of the spectral curve $\Gamma$. 

In the next examples we consider the differential field $K=\bbC (x)$ of rational functions with complex coefficients.


\begin{ex}\label{ex-alabeada}
Let us consider the differential operator in $\bbC(x)[\partial]$
\begin{equation}
     L =\partial^3 -\frac{6}{x^{2}} \partial +\frac{12}{x^{3}}+h ,
\end{equation}
obtained from formula \eqref{eq-BsqDGU-L3} with $ q_0 = \frac{6}{x^3}+h$ and $ q_1 = -\frac{6}{x^{2}}$, with $h\in \mathbb{C}$.

We start checking that $L$ is a Boussinesq differential operator and obtaining the vectors of constants for each branch.
The first Boussinesq systems satisfied in each branch $i=1,2$ are:
\[
\sBsq_{1,1}\left( 12/x^3 +h, -6/x^2 ,\tilde{c}_{1,1}\right)=0 , \tilde{c}_{1,1}=\left(-\frac43 h,0\right)
\]
and
\[
\sBsq_{1,2}\left( 12/x^3 +h, -6/x^2 ,\tilde{c}_{2,2}\right)=0, \tilde{c}_{1,2}=\left(0,-\frac53 h,0\right).
\]
Hence, by Lemma \ref{lemma-system}, $[L, P_4 ]=0$ and $[L,P_5 ]=0$ where $P_4$ and $P_5$ are given in  \eqref{defi-Bsq-formal-operators}, substituting $u_0$ by $ \tilde{u}_0=12/x^3 +h $ and $u_1$ by $ \tilde{u}_1= -6/x^2 $, with vectors of constants $\tilde{c}_{1,1}$ and $\tilde{c}_{1,2}$ respectively. 

\medskip

We are ready now to run the \SpF Algorithm. 
\begin{enumerate}
    \item The centralizer of $L$ equals $\bbC[L,A_1,A_2]$ with
    \begin{align*}
    A_1:=&P_4(\tilde{u_0},\tilde{u_1},\tilde{c}_{1,1}) = \partial^4 -\frac{8}{x^2 }\partial^2 +\frac{24}{x^3 }\partial -\frac{24}{x^4},\\
    A_2:=&P_5(\tilde{u_0},\tilde{u_1},\tilde{c}_{1,2}) =  \partial^5 -\frac{10}{x^2 }\partial^3 +\frac{40}{x^3 }\partial^2- \frac{80}{x^4 }\partial+\frac{80}{x^5 }.
    \end{align*}
    
    \item Using differential resultants we compute 
    \[f_1=-\mu^3+(\lambda-h)^4, \ f_2:=-\gamma^3-(h-\lambda)^5, \ f_3=\gamma^4-\mu^5.\]
    
    \item In this case, one can easily verify that $I(L)=(f_1,f_2,f_3)$ is a prime ideal. Moreover, the curve defined by $I(L)$ is a non-planar curve $\Gamma$ parametrized by $$\aleph (\tau)=(\tau^3+h,\tau^4,-\tau^5), \tau\in\bbC.$$ 
    This is the first explicit example of a non-planar spectral curve.
    
    \item The first differential subresultants of $L-\lambda$, $A_1-\mu$ and $A_2-\gamma$ pairwise are equal to $\phi_{i,0}+\phi_{i,1}\partial$, see \eqref{eq-phiij}, with
    \begin{align*}
        \phi_{1,0}=&(h -\lambda)\mu - \frac{4\mu}{x^3} +\frac{8(\lambda-h)}{x^4}, \ 
        \phi_{1,1}=(\lambda-h)^2-\frac{2\mu}{x^2}+4\frac{(\lambda-h)}{x^3},\\
        \phi_{2,0}=&(h -\lambda)^3 - \frac{4(h-\lambda)^2}{x^2} +\frac{8\gamma}{x^4}, \ 
        \phi_{2,1}=(\lambda-h)^3-\frac{4(h-\lambda)^2}{x^2}+\frac{8\gamma}{x^3},\\
        \phi_{3,0}=&-\gamma \left(\mu^2  + \frac{4\gamma}{x^3} -\frac{8\mu}{x^4}\right), \ 
        \phi_{3,1}=\mu^3 - \frac{2\gamma^2}{ x^2} +\frac{ 4\gamma\mu}{x^3}.
    \end{align*}
    \item $Z=\{(h,0,0)\}$ contains only the singular point of this curve.
    \item For every $P_0\in \Gamma\backslash Z$ then $P_0=\aleph(\tau_0)$, with $\tau_0\in\bbC$ and $\tau_0\neq 0$. 
    
    \item Then 
    \begin{equation}
        \phi_0=\phi(P_0)=\phi_i(\aleph(\tau_0))=\frac{-\tau_0^3x^3 + 2\tau_0^2 x^2 - 4\tau_0 x + 4}{(\tau_0^2 x^2 - 2\tau_0 x + 2)x}.
    \end{equation}
    
    \item For every nonzero $\tau_0\in\bbC$ the algorithm returns 
    \[[I(L),\aleph(\tau_0), \partial+\phi_0],\]
    with $L-(\tau_0^3 +h)= \left( \partial_0^2+\phi_0\partial+\phi_0^2+2\phi_0'-\frac{6}{x^2 } \right)\cdot (\partial+\phi_0)$ as in \eqref{eq-factorization}.
\end{enumerate}

\end{ex}

\begin{ex}\label{ex-factoring-curva-plana}
Let us continue with the example in \ref{ex-spectral-curve-Bsq} and \ref{ex-L3-first}.
Although the \SpF Algorithm would return that $ L =\partial^3 -\frac{15}{x^{2}} \partial +\frac{15}{x^{3}} +h$, "\emph{is not geometrically reducible}" we can modify the procedure to return a factorization of $L-\lambda_0$ for every $\lambda_0\in\bbC$. Recall that $\cC_K(L)$ is isomorphic to the ring of the plane affine algebraic curve $\Gamma$ defined by the ideal $I=(f_1)$ with 
\[f_1=\dres(L-\lambda,A_1-\mu)= -\mu^3+(h-\lambda)^4.\]
The curve $\Gamma$ is rational, with parametrization $\aleph(\tau)=(h-\tau^3, \tau^4)$, $\tau\in\bbC$.

The first differential subresultant of $L-\lambda$ and $A_1-\mu$ equals $\phi_{1,0}+\phi_{1,1}\partial$ with 
\begin{align*}
    \phi_{1,0}=&\mu(h-\lambda)- \frac{5\mu}{x^3} -\frac{20(h - \lambda)}{x^4} - \frac{300}{x^7}, \\
    \phi_{1,1}= &\frac{(h-\lambda)^2}{x}- \frac{5\mu}{x^2} - \frac{100}{x^6}.
\end{align*}
Observe that $\phi_{1,1}$ is never zero, not even at the singular point $(h,0)$ of $\Gamma$.

For every $P_0\in\Gamma$ there exists $\tau_0\in\bbC$, such that $P_0=\aleph(\tau_0)$.
Then
\[\phi_0=\frac{\phi_{1,0}(\aleph(\tau_0))}{\phi_{1,1}(\aleph(\tau_0))}=\frac{\tau_0^4x^4 + 5\tau_0^3x^3 + 15\tau_0^2x^2 + 30 \tau_0 x + 30 }{(\tau_0^3x^3 + 5\tau_0^2x^2 + 10\tau_0 x + 10)x}.\]
Hence, for every $\lambda_0\in\bbC$ the factorization of $L-\lambda_0$ in $\bbC(x)[\partial]$ is
\[
    L-\lambda_0 = \left( \partial_0^2+\phi_0\partial+\phi_0^2+2\phi_0'-\frac{15}{x^2 } \right) \left(  \partial+\phi_0 \right), \mbox{ with }\lambda_0 = h-\tau_0^3.
\]
\end{ex}

\section{Conclusions}

In this work we considered the  factorization problem of a third order ordinary differential operator $L-\lambda$, for a spectral parameter $\lambda$, whose coefficients belong to a differential field $K$ with field of constants $\bbC$. It is assumed that $L$ is Boussinesq operator, implying it is algebro-geometric over $K$ and guarantying a nontrivial centralizer. The centralizer is proved to be isomorphic to the ring of an affine algebraic curve, the famous spectral curve $\Gamma$. As far as we know, this is the first time that the ring structure of the centralizer has been taken into consideration, as the appropriate algebraic structure to allow the factorization of the third order operator $L-\lambda$.

Based on the nature of $\Gamma$, {which is proved to be in $\bbC^3$}, we give the symbolic algorithm \SpF Algorithm \ref{alg-SpF} to factor $L-\lambda_0$, for almost every point $P_0 = (\lambda_0 ,\mu_0 ,\gamma_0 )$ of the spectral curve, using differential subresultants. In this context, the first explicit example of a non-planar spectral curve arises, Example \ref{ex-alabeada}, as well as the  factorization it provides for $ L- \lambda_0$.  As far as we know, this is a new algorithm {specifically designed for the case of non-planar spectral curves that appear for Boussinesq operators}. Factorizations over planar spectral curves have been previously presented, for instance for second order operators \cite{MRZ1}, or fourth order operators with rank $2$ in \cite{PRZ2019}. In addition, Boussinesq operators provide examples of planar spectral curves as shown in Example \ref{ex-factoring-curva-plana}.

The present work is the natural continuation in a program dedicated to the factorization of rank $1$ algebro-geometric differential operators, that was already successful in the order $2$ case, \cite{MRZ1}. Our ultimate goal is an effective approach to the {\sl direct} spectral problem and the development of the appropriate {\sl spectral} Picard-Vessiot fields containing all the solutions of the operator $L-\lambda$. Spectral Picard-Vessiot fields were studied for Schrödinger operators in \cite{MRZ2}. The development of algorithms for their computation is part of an ongoing project that focuses on the study of the centralizers of operators with potentials verifying one of Gel'fand-Dickii integrable hierarchies \cite{GD}.

\section{Acknowledgments}

The authors would like to thank their colleagues (especially those from the Integrability seminar in Madrid) for interesting discussions that encouraged them to write this paper.
A sincere acknowledgement to Rafael Hernandez Heredero for providing the references on the Boussinesq hierarchy and for the enlightening conversations on this topic, that contributed to a better understanding of the mechanisms behind integrability. The authors really appreciate the exchange of ideas regarding the non-planar spectral curves with Emma Previato, we are very grateful for the conversations we had on spectral problems and for her constant encouragement to pursuit this project.

The first author is a member of the Research Group ``Modelos ma\-tem\'aticos no lineales". The second  author acknowledges the support of Grupo UCM 910444.

\bibliographystyle{plain}

\bibliography{Bibliography.bib}



\end{document}